\documentclass[11pt]{amsart}
\usepackage{amssymb}

\usepackage{lmodern}
\usepackage[T1]{fontenc}

\usepackage{mathrsfs}

\usepackage[headings]{fullpage}
\usepackage{paralist}
\usepackage{xspace}

\usepackage{color}
\usepackage[normalem]{ulem}

\usepackage[colorlinks,pagebackref=true]{hyperref}
\usepackage[alphabetic, abbrev, lite, backrefs]{amsrefs}

\makeatletter

\@addtoreset{equation}{section}
\def\theequation{\thesection.\@arabic \c@equation}

\def\@citecolor{blue}
\def\@linkcolor{blue}
\def\@urlcolor{blue}

\def\theenumi{\@alph\c@enumi}

\makeatother

\theoremstyle{plain}
\newtheorem{thm}[equation]{Theorem}
\newtheorem{lemma}[equation]{Lemma}

\newtheorem{propn}[equation]{Proposition}

\theoremstyle{definition}
\newtheorem{defn}[equation]{Definition}
\newtheorem{example}[equation]{Example}
\newtheorem{question}[equation]{Question}
\newtheorem{conj}[equation]{Conjecture}

\theoremstyle{remark}
\newtheorem{remark}[equation]{Remark}
\newtheorem{remarkwithoutbox}[equation]{Remark}
\newtheorem{discussion}[equation]{Discussion}
\newtheorem{construction}[equation]{Construction}

\newtheorem{notation}[equation]{Notation}
\newtheorem*{notation*}{Notation}

\let\oldendremark\endremark
\def\endremark{\hfill\qedsymbol\oldendremark}
\let\oldendexample\endexample
\def\endexample{\hfill\qedsymbol\oldendexample}
\let\oldendconstruction\endconstruction
\def\endconstruction{\hfill\qedsymbol\oldendconstruction}
\let\oldenddiscussion\enddiscussion
\def\enddiscussion{\hfill\qedsymbol\oldenddiscussion}

\newcommand{\frakp}{{\mathfrak p}}
\newcommand{\frakm}{{\mathfrak m}}
\DeclareMathOperator{\calO}{\mathscr O} \let\strSh\calO
\DeclareMathOperator{\calI}{\mathscr I}
\DeclareMathOperator{\PP}{\mathbb{P}} \let\projective\PP
\DeclareMathOperator{\ara}{ara} \let\arank\ara
\newcommand{\ol}{\overline}

\newcommand{\ints}{\mathbb{Z}}
\newcommand{\defeq}{:=}
\DeclareMathOperator{\lcm}{lcm}
\DeclareMathOperator{\Var}{Var}
\DeclareMathOperator{\coker}{coker}
\DeclareMathOperator{\affine}{\mathbb{A}}
\DeclareMathOperator{\projdim}{pd}
\DeclareMathOperator{\height}{ht}

\DeclareMathOperator{\charact}{char}
\DeclareMathOperator{\Supp}{Supp}
\DeclareMathOperator{\Spec}{Spec}
\DeclareMathOperator{\Proj}{Proj}
\DeclareMathOperator{\rhomo}{\widetilde H}

\DeclareMathOperator{\Min}{Min}

\DeclareMathOperator{\homology}{H}
\DeclareMathOperator{\homolSh}{\mathcal H}
\DeclareMathOperator{\cohdim}{cd}
\newcommand{\define}[1]{\emph{#1}}
\newcommand{\minus}{\ensuremath{\smallsetminus}}
\DeclareMathOperator{\gens}{Gen}
\DeclareMathOperator{\monom}{Mon}

\DeclareMathOperator{\qccd}{qccd}
\DeclareMathOperator{\etcd}{\acute{e}cd}
\newcommand{\etale}{\'etale\xspace}
\newcommand{\Etale}{\'Etale\xspace}
\newcommand{\met}{\ensuremath{\mathrm{\acute{e}t}}}

\def\to{\longrightarrow}

\title[\'Etale cohomological dimension and arithmetic rank]%
{\'Etale cohomological dimension, a conjecture of Lyubeznik 
and bounds for arithmetic rank}

\author{Manoj Kummini}
\address{Department of Mathematics, Purdue University, West Lafayette, IN
47907-2607}
\email{nkummini@math.purdue.edu}

\author{Uli Walther}
\address{Department of Mathematics, Purdue University, West Lafayette, IN
47907-2607}
\email{walther@math.purdue.edu}
\thanks{UW was supported by the NSF under grant DMS~0901123.}


\begin{document}

\begin{abstract}
We produce a criterion for open sets in projective $n$-space over a
separably closed field to have \'etale cohomological dimension bounded by
$2n-3$. We use the criterion to exhibit a scheme for which \'etale
cohomological dimension is smaller than what a conjecture of G.~Lyubeznik
predicts; the discrepancy is of arithmetic nature.
  
For a monomial ideal, we relate extremal graded Betti numbers and \'etale
cohomological dimension of the complement of the corresponding subspace
arrangement.
Moreover, we derive upper bounds for its arithmetic rank in terms of
invariants distilled from the lcm-lattice.
\end{abstract}

\maketitle

\section{Introduction}

Let $R$ be a Noetherian ring and $I$ an $R$-ideal. The
\emph{arithmetic rank} of $I$, denoted $\arank I$, is the least
integer $r$ for which there exist $g_1, \ldots, g_r \in R$
such that $\sqrt{(g_1, \ldots, g_r)} = \sqrt{I}$. Determining the exact
value of the arithmetic rank of an ideal is in general a very difficult
problem, and a complete answer is known only 
 in a small set of cases.
In a natural way, lower bounds for $\ara(I)$ come out of
non-vanishing results for suitable cohomology groups; on the other
hand, finding upper bounds usually requires explicit constructions.
The purpose of this paper is to provide some results of both kinds.

Let $U$ be a Noetherian $\Bbbk$-scheme. The \define{quasi-coherent
cohomological dimension} of $U$,
\begin{equation}
\label{equation:qccd}
\qccd (U) \defeq \max \{i : \homology^i(U, F) \neq 0, 
\;\text{for all quasi-coherent sheaves $F$ on $U$}\},
\end{equation}
is very closely connected to arithmetic rank. 
For example, suppose that $X$ is an affine $\Bbbk$-scheme with 
coordinate ring $R$. Let $Z$ be the subscheme of $X$ defined by an
$R$-ideal $I$ and
take $U = X \minus Z$. Then
\begin{equation}
\label{equation:arankLowerBddByCD}
\arank I \geqslant \qccd(U) + 1. 
\end{equation}

\smallskip

\noindent Let $U_\met$ be the \define{small \etale site} on $U$;
see~\cite {MilnEtale80}*{Section II.1}. By a \define{torsion sheaf on
  $U_\met$}, we mean a sheaf of Abelian torsion groups on $U_\met$ on
which $\charact \Bbbk$ operates as a unit.  Let $\ell \neq \charact
\Bbbk$ be a prime number; an \emph{$\ell$-torsion sheaf on $U_\met$}
is a torsion sheaf on which the action of $\ell$ is nilpotent.  The
\define{\etale cohomological dimension} $\etcd(U)$ and the
\define{\etale $\ell$-cohomological dimension} $\etcd_\ell(U)$ of $U$
are measures of the topological complexity of the underlying variety
\cite{LyubEtCohDim93}*{Sections~10 and~11}; they are defined by
\begin{align}
\begin{split}
\etcd(U) \defeq & \max \{i : \homology^i(U_\met, F) \neq 0, 
\;\text{for all torsion sheaves $F$ on $U_\met$} \}, 
\;\text{and}, \\ 
\etcd_\ell(U) \defeq & \max \{i : \homology^i(U_\met, F) \neq 0, 
\;\text{for all $\ell$-torsion sheaves $F$ on $U_\met$} \}.
\end{split}
\end{align}
It is known that $\etcd(U) = \max \{\etcd_\ell(U) : \ell \neq \charact
\Bbbk \;\text{is a prime number}\}$, and this number
gives another estimate for arithmetic rank:
\begin{equation}
\label{equation:arankLowerBddByCD2}
\arank I \geqslant \etcd(U) - \dim U + 1.
\end{equation}

Consider an $n$-dimensional scheme $U$ of finite-type over a
separably closed field $\Bbbk$. The most general vanishing result is:
$\etcd(U) \leqslant 2n$,~\cite{MilnEtale80}*{Theorem~VI.1.1}. There
are results refining this bound, if $U$ can be embedded in a nice ambient
space $X$, and if some information about $X\minus U$ is known. Such
information, typically, includes $\dim (X \minus U)$, the number of
components and the combinatorial data on their intersections.
For example, $\etcd(U)
\leqslant 2n-1$ if and only if no $n$-dimensional component of $U$ is
proper~\cite {LyubEtCohDim93}*{Theorem~3.5}. One step further, for
a non-singular affine variety $X$ and a closed subscheme $Z$ of
$X$,~\cite{LyubEtCohDim93}*{Theorem~4.9} shows that $\etcd(X\minus Z)
\leqslant 2n-2$ if and only if every irreducible component of $Z$ has
dimension $\geqslant 2$ and $Z \minus \{p\}$ is locally analytically
connected at $p$ for every closed point $p \in Z$. Suppose now in addition
that
each component of $Z$ has dimension at least $3$. Under what
circumstances can one hope to push the estimate for $\etcd(U)$ one
step further? Topological information on the
components of $Z$ together with combinatorial data regarding how they
are stacked together is not enough, see
Remark~\ref{rem:2vs3}.\eqref{rem:2vs3detail}. On the
positive side,
we give in Theorem~\ref{thm:subschLSTCI} a sufficient
condition to ensure that $\etcd(U) \leqslant 2n-3$, the additional
ingredient \ref{thm:subschLSTCI}.\eqref{enum:constecd} being of
arithmetic nature.

\smallskip

Both estimates~\eqref{equation:arankLowerBddByCD}
and~\eqref{equation:arankLowerBddByCD2} result from
covering $U$ by
$\arank I$ affine open subsets.
The inequalities
also hold when $X$ is a projective
$\Bbbk$-scheme, $Z$ is a closed subscheme and $U = X \minus
Z$. For, if
$X'$ and $Z'$ are the respective affine cones and
$U' = X' \minus Z'$, by~\cite {BrShLoco98}*{Chapter~20} 
and~\cite {LyubEtCohDim93}*{Proposition~10.1} (also see 
\eqref{equation:affineConeEtCD} below),
\begin{equation}
\label{equation:affineConeCDs}
\qccd(U') = \qccd(U) \;\text{and}\; \etcd(U') = \etcd(U)+1.
\end{equation}

The following conjecture of G.~Lyubeznik resulted from an attempt to
understand
whether
one of the estimates of~\eqref{equation:arankLowerBddByCD} and
\eqref{equation:arankLowerBddByCD2} provides
a better bound than the other, in general:

\begin{conj}[\cite {LyubSurveyLocalCoh02}*{p.~147}]
\label{conj:lyubeznikECD}
Let $U$ be a $\Bbbk$-scheme. Then $\etcd(U) \geqslant \qccd(U) + \dim U$. 
\end{conj}
Conjecture~\ref{conj:lyubeznikECD}
holds for
\begin{inparaenum}
\item complements of certain determinantal
varieties, if $\charact \Bbbk > 0$~\cite {BrunsSchwanzlDetVar90}*{p.440},
\item \label{enum:OgusExample} complements of certain Segre embeddings, in
characteristic zero~\cite{OgusLocCohDim73}*{Example~4.6}, and,
\item  complements of smooth
projective varieties, in characteristic zero~\cite
{VarbaroArithRk10}*{Theorem~2.21}.
\end{inparaenum}
However, 
in Example~\ref{example:reisnerLyubeznik}, 
we combine bounds for \etale cohomological dimension from
Theorem~\ref{thm:subschLSTCI} with arithmetic information from
Theorem~\ref{thm:constEtcd} and
provide a
counterexample to Conjecture~\ref{conj:lyubeznikECD}. 

\medskip

The \define{cohomological dimension} of $I$ on $R$ is defined by:
\begin{equation}
\label{equation:cohdim}
\cohdim(I,R) \defeq \max \{i : \homology^i_I(M) \neq 0, \;\text{for
  all $R$-modules $M$}\}. 
\end{equation}
Suppose that either $X = \Spec R$ is affine or that $R$ is graded, $X =
\Proj R$ and $I$ is homogeneous. Let $Z$ be defined by $I$.  Then
\begin{equation}
\label{equation:cohdimqccd}
\cohdim(I,R) = \qccd(X \minus Z)+1
\end{equation}
In the affine case,
see~\cite {Loco24Hrs}*{Proposition 9.12}; the projective case follows from
it, using~\eqref{equation:affineConeCDs}.

\medskip

Section~\ref{sec:upperbounds} deals with obtaining upper bounds for
arithmetic rank for monomial ideals. Let 
$R = \Bbbk[x_0, \ldots, x_n]$; write
$\mathfrak m$ for the homogeneous maximal ideal of $R$.
Let $I$ be a homogeneous $R$-ideal.
In this case, we may also consider the \emph{homogeneous arithmetic
rank} of $I$, which is the least integer $r$ such that there exist
homogeneous polynomials $g_1, \ldots, g_r \in R$ such that $\sqrt{(g_1,
\ldots, g_r)} = \sqrt{I}$. We denote it by $\arank_h I$. While $\arank I$
gives the least number of equations to define $\Spec R/I \subseteq
\affine^{n+1}$ set-theoretically, $\arank_h I$
gives the least number of equations to define $\Proj R/I \subseteq
\projective^n$ set-theoretically.  Note that
$\arank I \leqslant \arank_h I$; whether equality holds always is unknown.

\begin{defn}
\label{defn:lcmLattice}
We write $\monom(R)$ for the set of monomials of $R$.
Suppose that $I$ is a monomial $R$-ideal. We denote the set of the (unique)
minimal monomial generators of $I$ by $\gens(I)$.
The \define{lcm-lattice} of $I$, denoted $L_I$, is the set of the least
common multiples of the elements of $\gens(I)$, partially ordered by
divisibility.
The unique minimal element of $L_I$, which we denote by $\hat 0$, is the
monomial $1$. Let $L$ be any finite lattice.
For $\tau, \sigma \in L$, let $[\tau, \sigma] \defeq \{\rho : \tau
\preccurlyeq \rho \preccurlyeq \sigma\}$ and $(\tau, \sigma) \defeq \{\rho
: \tau \precneqq \rho \precneqq \sigma\}$, both inheriting the order from
$L$.  For $\sigma \in L$, the \emph{height} of $\sigma$, written $\height_L
\sigma$, is the length of the longest (saturated) chain in $[\hat 0,
\sigma]$.  The \emph{height} of $L$,
written $\height L$, is the length of the longest (saturated) chain in $L$.
The elements of $L$ of height $1$ are called \define{atoms}.
\end{defn}

The lcm-lattice $L_I$ is an atomic lattice; the 
atoms of $L_I$ are
precisely the monomial generators of $I$. From $L_I$ we construct
homogeneous generators $g_i$ of $I$ up to radical. This yields
two upper bounds.
First, in Theorem~\ref{thm:araLeqDimRootingCx} we show that
$\arank_h(I)-1$ is bounded above by the dimension of any rooting
complex (Definition~\ref{defn:rootingmaps}) on $L_I$, generalizing a
theorem of K.~Kimura~\cite{KimuraLyubResARA09}*{Theorem~1}. Later, in
Proposition~\ref{thm:araLeqHtLcmLattice}, we show that $\arank_h(I) -1$ is
bounded above by the height of $L_I$.

\section{\Etale cohomological dimension}

We look first at \etale cohomological dimensions of open subschemes of
$\projective^n_\Bbbk$. Then we express the \etale cohomological dimension
of complements of affine coordinate subspace arrangements in terms of
extremal Betti numbers of their defining ideals.  We combine these results
to construct a counterexample to a conjecture of Lyubeznik. 

\begin{notation}
\label{notation:notationECD}
Throughout this section, we take $\Bbbk$ to be separably closed. Let
$\ell \neq \charact \Bbbk$ be a prime number.
We will abbreviate `a sheaf of $\ints/\ell$-modules' by `a $\ints/\ell$-module'.
We will omit the subscript $\met$ in the proofs in this section, since all
the cohomology groups in the proof are for (small) \etale sites. For any
scheme $X$ and a point $p \in X$, write $\dim p = \dim \ol{\{p\}}$.
For a real number $r$, $\lfloor r\rfloor$ denotes the unique integer such
that $ r -1 < \lfloor r\rfloor \leqslant r$. For a local ring $A$, we write
$A^{\mathrm{sh}}$ for its strict Henselization.
If $j: U \to X$ is an \etale morphism of finite type and $F$
is a sheaf on $X_\met$, then we write $F|_U$ for $j^*F$~\cite
{MilnEtale80}*{Remark~II.3.1.(a)}.
\end{notation}

\subsection*{Open subschemes of $\mathbf P^n$}
We find an upper bound for the \etale cohomological dimension of the
complements of certain projective schemes. Call an irreducible
subvariety $Z$ of $\projective^n_\Bbbk$, defined by a prime $R$-ideal $I$,
\define{analytically irreducible at the vertex} if 
$I\widehat{R_{\mathfrak m}}$
(where $\;\widehat{}\;$ denotes completion) is a prime ideal.

\begin{thm}
\label{thm:subschLSTCI}
Let $Z \subseteq \projective^n_\Bbbk$ be a closed subscheme without any
zero-dimensional components, defined by the ideal sheaf $\calI$.
Suppose that
\begin{compactenum}
\item \label{enum:localnminustwo} In the stalk $\calO_{\PP^n,p}$, 
$\arank \calI_p \leqslant \min \{n-2,
n- \dim p - 1\}$
for all $p \in Z$ 
different from the generic points of $Z$,
\item \label{enum:nonemptyintersections} the irreducible components of $Z$
are analytically irreducible at the vertex and have
pairwise non-empty intersections, and,
\item \label{enum:constecd} 
$\homology^{\geqslant 2n-2}((\projective^n_\Bbbk \minus Z)_\met,
\ints/\ell) = 0$. 
\end{compactenum}
Then $\etcd_\ell(\projective^n_\Bbbk \minus Z) \leqslant 2n-3$.
\end{thm}

See Remark~\ref{remark:hypForProjComplement} below for a discussion of the
hypotheses of Theorem~\ref{thm:subschLSTCI}. The key step in the proof of
the theorem is Lemma~\ref{lemma:vertexIrredHypers}, which establishes that
certain local \etale cohomology sheaves on the affine cones are not
supported at the vertex. Using this result, we get a bound for the \etale
cohomological dimension of closed subschemes of  $(\projective^n_\Bbbk
\minus Z)$, in Lemma~\ref{propn:constantsOnClosed}. We start by
collecting some basic facts of \etale cohomology.

\begin{discussion}
\label{discussion:filtrConstLowerShr}
Let $X$ be a scheme, and $Z$ a closed subscheme of $X$; denote
the complementary open embedding by
$j : (X \minus Z) \to X$. We write $\homolSh^j_{Z}(X,-)$ for the $j$th
right derived
functor of $F \mapsto \ker(F \to j_*j^*F)$. Let $F$ be any $\ints/\ell$-module
on $X$. First, \cite
{LyubEtCohDim93}*{Corollary~2.8} gives 
\begin{equation}
\label{equation:lyubezniknhalfj}
\dim \Supp \homolSh^j_{Z}(X, F) \leqslant 
\left\lfloor \dim \Supp F - \frac j2 \right\rfloor, 
\quad\text{for all}\; j.
\end{equation}
Secondly, let $p \in Z$ and write $A = (\strSh_{X, p})^{\mathrm{sh}}$.  
Then, by~\cite {LyubEtCohDim93}*{(1.11b)},
\begin{equation}
\label{equation:stalkHomolSheaf}
\homolSh^i_{Z}(X, F)_p = \homology^i_{Z}(
\Spec (\strSh_{X, p})^{\mathrm{sh}}, 
F|_{\Spec (\strSh_{X, p})^{\mathrm{sh}}}).
\end{equation}
Further, we have a spectral sequence
(see~\cite {LyubEtCohDim93}*{(1.3a)})
\begin{equation}
\label{equation:spectralseq}
E_2^{ij} = 
\homology^i(X, 
\homolSh^j_{Z}(X, F)) \Longrightarrow \homology^{i+j}_{Z'}(X, F).
\end{equation}
If $X = \Spec A$, where $A$ is a strictly Henselian ring, then
we have $E_2^{ij} = 0$ for all $i > 0$, so, 
\begin{equation}
\label{equation:spectralseqSHRing}
\homology^0(\Spec A, \homolSh^j_{Z}(\Spec A, F)) = 
\homology^j_{Z}(\Spec A, F), \;\text{for all}\; j.
\end{equation}
If, additionally, $A$ is the strict Henselization of a ring that is
essentially of finite type over a field, and $Z$ is
defined an ideal $I$, then (by \cite{LyubEtCohDim93}*{Proposition~2.9})
\begin{equation}
\label{equation:ecdAra}
\homology^j_{Z}(\Spec A, F) = 0,
\quad\text{for all}\; j > \dim A + \arank I, 
\;\text{and for all}\; F.
\end{equation}
Now suppose that $X$ and $Z$ are projective, and write $X'$ and $Z'$ for
their respective affine cones.
Then, by~\cite {LyubEtCohDim93}*{Proposition~10.1},
\begin{equation}
\label{equation:affineConeEtCD}
\homology^{\geqslant t+1}(X' \minus Z', \pi^*F) = 0
\;\text{if and only if}\;
\homology^{\geqslant t}(X \minus Z, F) = 0 
\end{equation}
for every sheaf $F$ on $(X \minus Z)_\met$ and for every $t \geqslant 0$.
\end{discussion}

\begin{notation}
We write $I$ for the defining ideal of $Z$ and
$\Var(I)$ for $\{\mathfrak p
\in \Spec R : I \subseteq \mathfrak p\}$.
\end{notation}

\begin{lemma}
\label{lemma:genericPtAra}
Assume hypothesis~\ref{thm:subschLSTCI}.\eqref{enum:localnminustwo}. 
Then $\arank I R_{\mathfrak p}\leqslant n-2$
for all $\mathfrak p \in\Var(I)\minus\{\frakm\}$. Moreover, 
$\arank I R_{\mathfrak p}\leqslant \height \mathfrak p - 1$ for all 
$\mathfrak p \in \Var(I) \minus (\{ \frakm \} \cup \Min(I))$.
\end{lemma}

\begin{proof}
Take $\mathfrak p \neq \frakm \in \Var(I)$ and write $\mathfrak
p^*\neq\frakm$ for the largest homogeneous $R$-subideal of $I$.
We denote the corresponding point in $Z$ by
$p$.  Since $\calO_{\PP^n,p}=\{\frac{f}{g}\mid g\not\in \frakp;
f,g\text{ homogeneous of equal degree}\}$ (modulo the usual
equivalence relation), there is a natural map $\calO_{\PP^n,p}\to
R_\frakp$ under which $\calI_p$ extends to $IR_\frakp$.  In
particular,
\begin{equation}\label{eqn-ara}
\ara(IR_\frakp)\leqslant
\ara(\calI_p)\qquad \text{ for all
}\frakp\in\Var(I)\minus\{\frakm\}. 
\end{equation}

If $p$ is not a generic point of $Z$, then \eqref{eqn-ara}, together
with hypothesis \ref{thm:subschLSTCI}.\eqref{enum:localnminustwo}, implies that $\ara(I R_\frakp) \le
\min\{n-2,n-\dim p-1\}=\min\{n-2,\height\frakp^*-1\}\le
\min\{n-2,\height\frakp-1\}$.  

So from now on assume $p\in Z$ to be generic (hence not closed), 
and let
$p'\in\ol {\{p\}} \subseteq Z$ be a point of dimension $\dim p -1$.  Since
$\calI_p=\calI_{p'}\calO_p$,
$\ara(\calI_p)\le\ara(\calI_{p'})$. 
Applying \eqref{eqn-ara} and hypothesis 
\ref{thm:subschLSTCI}.\eqref{enum:localnminustwo},
we find 
\[
\ara(IR_\frakp)\leqslant
\ara(\calI_p)\le\ara(\calI_{p'})\le \min\{n-2,n-\dim
p'-1\}\le n-2\qquad \text{ for all }\frakp\in\Var(I)\minus\{\frakm\}.
\] 
Finally, if $\frakp\not\in\Min(I)$ and
$p$ is generic,
$\height(\frakp)=\height(\frakp^*)+1$
(see~\cite{MatsCRT89}*{Exercise~13.6}).
Thus, in this case we find $\ara(I
R_\frakp)\le \min\{n-2,n-\dim
p'-1\}=\min\{n-2,\height\frakp^*\}= \min\{n-2,\height\frakp-1\}$.
\end{proof}

\begin{notation}
\label{notation:vertexIrredHypers}
Let $Y \subseteq \projective^n_\Bbbk$ be an irreducible and reduced
hypersurface such that $Z \subseteq Y$.
Let $Z' \subseteq Y'$ (where $Z' = \Spec R/I$) be the affine cones over $Z$
and $Y$, respectively.
Write $v$ for the vertex of $Y'$ (and of $Z'$),
and $(A, \mathfrak m_A) = (\strSh_{Y', v})^{\mathrm{sh}}$.
\end{notation}

\begin{lemma}
\label{lemma:vertexHypotheses}
Adopt Notation~\ref{notation:vertexIrredHypers}.
The hypotheses 
\ref{thm:subschLSTCI}.\eqref{enum:localnminustwo} 
and
\ref{thm:subschLSTCI}.\eqref{enum:nonemptyintersections}
imply, respectively: 
{\makeatletter\def\theenumi{\@alph\c@enumi${}^\prime$}\makeatother
\begin{asparaenum}
\item\label{enum:localnminustwoII}
$\arank I A_{\mathfrak q} \leqslant \min \{n-2, \height \mathfrak q-1\}$,
for all $\mathfrak q \in \Var(I) \minus (\{\mathfrak m_A\} \cup \Min (IA))$.  
\item\label{enum:nonemptyintersectionsII} for all minimal primes
$\mathfrak p$ and $\mathfrak q$ over $IA$, $\dim A/(\mathfrak p + \mathfrak
q) > 0$. 
\end{asparaenum} }
\end{lemma}

\begin{proof}
Write $(S, \mathfrak m_S) = \strSh_{Y', v}$. 
\underline{\eqref{enum:localnminustwoII}}:
Let $\mathfrak q \in \Var(I) \minus (\{\mathfrak m_A\} \cup \Min (IA))$. 
Write $\bar{\mathfrak  q} = \mathfrak q \cap S$. Since $A_{\mathfrak
q}$ is a localization of $(S\minus \bar{\mathfrak q})^{-1}A$, it suffices
to show that $\arank I S_{\bar{\mathfrak  q}} \leqslant \min \{n-2,
\height \mathfrak q -1\}$. Lemma~\ref{lemma:genericPtAra} gives
$\arank I S_{\bar{\mathfrak  q}} \leqslant n-2$. 
By the going-down property (see~\cite {MatsCRT89}*{15.1}) of the flat
extension $S \to A$, $\height \bar{\mathfrak q} \leqslant \height \mathfrak
q$. Hence, if $\bar{\mathfrak q} \not \in \Min(IS)$, then (again by 
Lemma~\ref{lemma:genericPtAra})  $\arank IS_{\bar{\mathfrak q}} \leqslant
\height \mathfrak q -1$.
If $\bar{\mathfrak q} \in \Min(IS)$, then $\height \bar{\mathfrak q} <
\height \mathfrak q$, so, 
$\arank I S_{\bar{\mathfrak  q}} \leqslant \height \bar{\mathfrak q}
\leqslant \height \mathfrak q -1$.

\underline{\eqref{enum:nonemptyintersectionsII}}:
Since $\Bbbk$ is separably closed, $A$ is, in fact, the Henselization
of $S$.
If $\mathfrak p$ is minimal over $IA$, then $\mathfrak p =
(\mathfrak p \cap S)A$,
by~\ref{thm:subschLSTCI}.\eqref{enum:nonemptyintersections}, since $\widehat{A} = 
\widehat S$.
Hence, for $\mathfrak p$ and $\mathfrak q$ be minimal over $IA$, 
\[
A/(\mathfrak p + \mathfrak q) = 
\left(\frac{S}{(\mathfrak p \cap S) + (\mathfrak q \cap
S)}\right)^\mathrm{sh},
\]
which has positive dimension, again 
by~\ref{thm:subschLSTCI}.\eqref{enum:nonemptyintersections}.
(Note that taking the strict Henselization commutes with going modulo an
ideal~\cite {MilnEtale80}*{p.~38}.)
\end{proof}

\begin{lemma}
\label{lemma:vertexIrredHypers}
Adopt Notation~\ref{notation:vertexIrredHypers}. Then, 
under the hypotheses 
\ref{thm:subschLSTCI}.\eqref{enum:localnminustwo} and
\ref{thm:subschLSTCI}.\eqref{enum:nonemptyintersections},
\[
\homolSh^{\geqslant 2n-1}_{Z'}(Y', F)_v = 0, \;\text{for all}\; F.
\]
\end{lemma}

\begin{proof}
In light of~\eqref{equation:stalkHomolSheaf} 
and~\eqref{equation:spectralseqSHRing},
it suffices to show that $\homolSh^{\geqslant 2n-1}_{Z'}(\Spec A, F) = 0$,
which we will achieve using~\cite {LyubEtCohDim93}*{Theorem~4.5}. Instead
of using 
\ref{thm:subschLSTCI}.\eqref{enum:localnminustwo} and
\ref{thm:subschLSTCI}.\eqref{enum:nonemptyintersections}, we will use the
corresponding statements from Lemma~\ref{lemma:vertexHypotheses}.

The embedding dimension of $A$ is at most $n+1$.
From Lemma~\ref{lemma:vertexHypotheses}.\eqref{enum:localnminustwoII}, we
see that every irreducible component
of $Z'$ has dimension at least $3$, so $c(I) \leqslant n-2$, in the notation of
\cite{LyubEtCohDim93}*{Theorem~4.5}.
We want to verify its hypotheses, taking $t=2n-1$. 

\underline{(i)}: the condition can rephrased as follows: for all $j$, 
$\dim \Supp \homolSh^j_{Z'}(\Spec A, F) \leqslant 
\left\lfloor \frac{t-j+1}2 \right\rfloor$,
if $t-j \leqslant 2c(I)-1$.
This, in turn,
follows from~\eqref{equation:lyubezniknhalfj}, 
since $\dim \Supp F \leqslant \dim A = n = \frac{t+1}2$.

\underline{(ii)}: Let $s$ be an odd integer such that $1 \leqslant s \leqslant
2c(I)-1$. If $j \geqslant t-s+1$, then,
by~\eqref{equation:lyubezniknhalfj}, we see that
$\dim \Supp \homolSh^j_{Z'}(\Spec A, F) \leqslant \lfloor \frac s2 \rfloor
= \frac{s+1}2 -1 $, so 
$\homolSh^j_{Z'}(\Spec A, F)$ is not supported at geometric points of
dimension $\frac{s+1}2$.
Hence assume that $j=t-s$. 
Let $q \in Z'$ be such that $\dim q = \frac{s+1}2$,
defined by a prime $A$-ideal $\mathfrak q$ such that $A/(\mathfrak q +
\mathfrak q') = 0$ for some prime $A$-ideal $\mathfrak q'$ that is minimal
over $IA$. We need to show that $\homolSh_{Z'}^{t-s}(\Spec A,F)_q = 0$.
It follows from 
Lemma~\ref{lemma:vertexHypotheses}.\eqref{enum:nonemptyintersectionsII} 
that $\mathfrak q$ is not a minimal prime over $IA$. Write $B = (A_{\mathfrak
q})^{\mathrm{sh}}$. By
Lemma~\ref{lemma:vertexHypotheses}.\eqref{enum:localnminustwoII}, 
$t-s = (2n-1) - (2 \dim q -1) =  2 \dim B > \dim B + \arank IB$. 
By~\eqref{equation:stalkHomolSheaf},
and~\eqref{equation:ecdAra}
$\homolSh_{Z'}^{t-s}(\Spec A,F)_q = 
\homology_{Z'}^{t-s}(\Spec B,F|_{\Spec B}) = 
0$.

\underline{(iii)}: 
Note that it suffices to check for $j=t-s$ ($q=t-s$, 
in the notation of \cite{LyubEtCohDim93}*{Theorem~4.5}.)
Now use~\eqref{equation:lyubezniknhalfj}, 
and the fact that for all $s \geqslant 2c(I)$, $s-c(I) \geqslant 
\left\lfloor \frac{s+1}2\right\rfloor = 
\left\lfloor n- \frac{t-s}2 \right\rfloor$.
\end{proof}

\begin{lemma}
\label{propn:constantsOnClosed}
Under the hypotheses 
\ref{thm:subschLSTCI}.\eqref{enum:localnminustwo} and
\ref{thm:subschLSTCI}.\eqref{enum:nonemptyintersections},
$\etcd_\ell(Y \minus Z) \leqslant 2n-4$ 
for all closed subschemes $Y \subsetneq \projective^n_\Bbbk$.
\end{lemma}

\begin{proof}
We may assume that $Y$ is an irreducible and reduced hypersurface that
contains $Z$ (see \cite {MilnEtale80}*{VI.1.1}, and the first
paragraph of its proof) and that that $Z$ is reduced. Now adopt 
Notation~\ref{notation:vertexIrredHypers}.

Let $G$ be any $\ints/\ell$-module on $(Y \minus Z)$. We need
to show that $\homology^{\geqslant 2n-3}((Y \minus Z), G) = 0$.
By~\eqref{equation:affineConeEtCD}
it suffices
to show that
$\homology^{\geqslant 2n-2}(Y' \minus Z', \pi^*G) = 0$. Let $F$ be a sheaf of
$\ints/\ell$-modules on $Y'$ such that $F|_{(Y' \minus Z')}  = 
\pi^*G$. Since $Y'$ is affine, 
$\homology^{\geqslant 2n-2}(Y' \minus Z', \pi^*G) = 0$ 
if and only if $\homology^{\geqslant 2n-1}_{Z'}(Y', F) = 0$; see \cite
{MilnEtale80}*{III.1.25}.

We obtain, from Lemma~\ref{lemma:genericPtAra}, that
if $p \in Z'$ and $p \neq v$ then 
$\arank I (\strSh_{Y', p})^{\mathrm{sh}} \leqslant n-2$.
Therefore, by~\eqref {equation:ecdAra}, for all $p \neq v\in Y'$, 
$\homolSh^j_{Z'}(Y', F)_p = 0$ for all $j > 2n-\dim p-2$. 
In particular, the stalks $\homolSh^{\geqslant 2n-1}_{Z'}(Y', F)_p$ are zero,
for all $p \neq v \in Y'$. Combining this 
with Lemma~\ref{lemma:vertexIrredHypers} 
and~\eqref{equation:lyubezniknhalfj}, 
we see that
\[
\homolSh^{\geqslant 2n-1}_{Z'}(Y', F) = 0 \;\text{and}\;
\dim \Supp \homolSh^j_{Z'}(Y', F) \leqslant 
\min\left\{\left\lfloor n - \frac j2 \right\rfloor\!,
2n-j-2\right\}, \;\text{for all}\; j \leqslant 2n-2.
\]
%
%
%
%
%
Since $Y'$ is affine, $\homology^i(Y', \homolSh^j_{Z'}(Y', F)) = 0$, for
all $i > \dim \Supp \homolSh^j_{Z'}(Y', F)$.
Using~\eqref{equation:spectralseq}, we conclude that 
$\homology^{\geqslant 2n-1}_{Z'}(Y', F) = 0$.
\end{proof}

\begin{discussion}
\label{discussion:constructible}
In order to determine \etale cohomological dimension, it suffices to
consider constructible sheaves. Suppose that $X$ is Noetherian.
A sheaf $F$ (of $\ints/\ell$-modules, for us) on $X$ is
\define{constructible} if every irreducible closed subscheme $Z$ of $X$
contains a nonempty open subscheme $U$ such that $F|_U$ is locally constant
and has finite stalks~\cite{MilnEtale80}*{V.1.8(b)}. 
Let $F$ be a constructible $\ints/\ell$-module on $X_\met$. Then $F$
has a finite filtration by subsheaves such that successive quotients are of
the form $j_!G$ where $j : U \to X$ is the inclusion of an irreducible,
locally closed and constructible subset, and $G$ is a locally constant and
constructible $\ints/\ell$-module on $U$; see \cite {SGA4}*{IX.2.5}, using
\cite {SGA4}*{IX.2.4.(i)} and \cite{MilnEtale80}*{V.1.8(b)}. 
\end{discussion}

\begin{proof}[Proof of Theorem~\ref{thm:subschLSTCI}] 
Write $U = \projective^n \minus Z$ and $r = \etcd_\ell(U)$. We want to show
that $r \leqslant 2n-3$.
Let $F$ be a
sheaf of $\ints/\ell$-modules on $U$ such that $\homology^r(U, F) \neq 0$. We
may assume that $F$ is constructible. By induction on the length of the
filtration in Discussion~\ref{discussion:filtrConstLowerShr}, we may assume
that $F = j_!G$ where $j : V \to U$ is the inclusion of an irreducible
locally closed subset, 
and $G$ is a locally constant and
constructible $\ints/\ell$-module on $V$. If $j$ is not an open immersion,
then either $\dim \Supp F \leqslant n-2$ or $\dim \Supp F = n-1$ and (since
$\dim Z \geqslant 2$ by~\ref{thm:subschLSTCI}.\eqref{enum:localnminustwo}) every
maximum-dimensional component of $\Supp F$ is non-proper. In either case,
$\etcd(U) \leqslant 2n-3$. (The latter case follows from~\cite
{LyubEtCohDim93}*{Theorem~3.5}.) Therefore we may assume that $j$ is an
open immersion.

Write $\eta$ for the geometric generic point of $U$ and $V$, and let $i :
\eta \to V$ be the inclusion map. Write $F_1  = \ker (F \to i_*i^*F)$ and
$F_2 = \coker (F \to i_*i^*F)$, and consider the exact sequence
\[
0 \to j_!F_1 \to j_!F \to j_!i_*i^*F \to j_!F_2 \to 0.
\]
on $U$. (Note that $j_!$ is exact~\cite {MilnEtale80}*{II.3.1.4}.)
Since the stalks $(j_!F)_\eta$ and $(j_!  i_*i^*F)_\eta$ are
identical, $j_!F_1$ and $j_!F_2$ are supported in dimension at most
$n-1$.  By Proposition~\ref{propn:constantsOnClosed},
$\homology^{\geqslant 2n-3}(U, j_!F_1) = \homology^{\geqslant 2n-3}(U,
j_!F_2) = 0$. Hence, to show that $r \leqslant 2n-3$, it suffices to
show that $\homology^{\geqslant 2n-2}(U, j_! i_*i^*F) = 0$.
Observe that $i_*i^*F$ is a constant sheaf on
$V$. Hence, without loss of generality, $F = j_!(\ints/\ell)$.

Let $\iota$ be the inclusion map $\eta \to U$, $F_1 = 
\ker (F \to \iota_*\iota^*F)$ and $F_2 =
\coker (F \to \iota_*\iota^*F)$.
Consider the exact sequence
\[
0 \to F_1 \to F \to \iota_*\iota^*F \to F_2 \to 0.
\]
Note that $\iota_*\iota^*F$ is a constant $\ints/\ell$-module on $U$ and that
$F_1$ and $F_2$ are supported in dimension at most $n-1$. By
Lemma~\ref{propn:constantsOnClosed}, 
$\homology^{\geqslant 2n-3}(U, F_1) = 
\homology^{\geqslant 2n-3}(U, F_2) = 0$, while
by
hypothesis~\ref{thm:subschLSTCI}.\eqref{enum:constecd}, 
$\homology^{\geqslant 2n-2}(U, \iota_*\iota^*F) = 0$.
Hence $\homology^{\geqslant 2n-2}(U, F) = 0$.
\end{proof}

\subsection*{Affine coordinate subspace arrangements}

We now turn to the complements of affine coordinate subspace arrangements.
We will relate their \etale cohomology groups to graded Betti numbers of
the ideals defining the arrangements. Then the \etale cohomological
dimension can be expressed in terms of extremal Betti numbers.

\begin{notation}
\label{notation:AffineECD}
Let $S = \ints[x_0, \ldots, x_n]$ and let $\mathfrak a$ be a square-free
monomial $S$-ideal. Let $R = \Bbbk[x_0, \ldots, x_n]$ (as usual) and $R' =
\ints/\ell[x_0, \ldots, x_n]$. Let $I = \mathfrak aR$ and 
$J = \mathfrak aR'$.
Let $Z$ be the closed
$\Bbbk$-subscheme of $\affine^{n+1}_\Bbbk$ defined by $I$. 
\end{notation}

\begin{thm}
\label{thm:constEtcd}
With the notation as above,
\[
\max\{r: \homology^r((\affine^{n+1}_\Bbbk \minus Z)_\met, \ints/\ell) \neq 0\} = 
\max \{i+j : \beta^{R'}_{i,j}(J) \neq 0\}
\]
\end{thm}

\begin{notation}
Let $\mathcal A$ be the coordinate subspace arrangement defined by $I$
inside $\Bbbk^{n+1}$. 
Let $(L_{\mathcal A}, \preceq) $ be the intersection lattice of $\mathcal
A$, partially ordered by \emph{inclusion}: $v \preceq w$ if $v \subseteq
w$.
By $(-)^\vee$, we denote the operation of taking the Alexander dual of a
square-free monomial ideal.
Let $\mu : L_{\mathcal A} \longrightarrow L_{I^\vee}$ be a map of posets
with $\mu(v) = \prod_{x_i(v) = 0} x_i$. (See Definition~\ref{defn:lcmLattice}
for the lcm-lattice $L_{I^\vee}$ of $I^\vee$.)
Let $(P, \preceq)$ be a poset. Write $\Delta(P)$ for its order complex,
\textit{i.e.}, the simplicial complex on $P$ with faces $\{v_1, \ldots,
v_r\}$ whenever $v_1 \precneqq v_2 \precneqq \cdots \precneqq v_r$. For $v
\in P$, write $P_{\succeq v}$ for the induced poset on $\{w \in P : v
\preceq w\}$, and $P_{\succ v}$ for the induced poset on $\{w \in P : v
\precneqq w\}$.
Write $U = \affine^{n+1}_\Bbbk \minus Z$. 
\end{notation}

The following was, perhaps, first observed by V.~Gasharov, I.~Peeva and
V.~Welker~\cite{GPWSubspArr02}.

\begin{lemma}
\label{lemma:dualityLattices}
The map $\mu$ defines a
duality between $L_{\mathcal A}$ and $L_{I^\vee}$.
\end{lemma}

\begin{proof}[Sketch] (See, also,~\cite {GPWSubspArr02}*{Proof of
Theorem~3.1}.) The coatoms of $L_{\mathcal A}$ correspond to the
irreducible components of $Z$, which correspond, under $\mu$, to the
minimal generators of $I^\vee$, \textit{i.e.}, the atoms of 
$L_{I^\vee}$. The flats of $L_{\mathcal A}$ are
intersections of the coatoms; this corresponds to taking the lcm of the
various subsets of $\gens(I)$.
\end{proof}

\begin{defn}[\cite {BCPExtremal99}*{p.~498}]
A Betti number $\beta_{i,j}(M) \neq 0$ is \define{extremal} if
$\beta_{k,l}(M) = 0$ for all $(k,l)$ such that $(k,l) \neq (i,j)$, $k
\geqslant
i$ and $j-i \geqslant l-k$.
\end{defn}

\begin{proof}[Proof of Theorem~\ref{thm:constEtcd}]
Write $L = L_{\mathcal A}$ and $L' = L_{I^\vee}$. For $v \in L$, write
$d(v) = \dim_\Bbbk v$.
Then,
\[
\homology^r(U_\met, \ints/\ell) = \bigoplus_{v \in L}
\homology^{2n-r-2d(v)-1}\left(\Delta(L_{\succeq v}), \Delta(L_{\succ v});
\ints/\ell\right).
\]
This is~\cite {YanEtaleGorMcp00}*{Corollary~2, p.~311}, taken along with the
discussion in \cite {YanEtaleGorMcp00}*{Section~2, pp.~306-7}.
From the long exact sequence (we suppress the group
$\ints/\ell$ of coefficients)
\[
\cdots \rightarrow 
\homology^r(\Delta(L_{\succeq v}), \Delta(L_{\succ v}))
\rightarrow
\homology^r(\Delta(L_{\succeq v})) \rightarrow 
\homology^r(\Delta(L_{\succ v})) \rightarrow
\homology^{r+1}(\Delta(L_{\succeq v}), \Delta(L_{\succ v}))
\rightarrow \cdots
\]
and the fact that $\Delta(L_{\succeq v})$ is a cone, we obtain the
following:
\[
\homology^{r+1}(\Delta(L_{\succeq v}), \Delta(L_{\succ v}))
\simeq
\rhomo^r(\Delta(L_{\succ v})), \;\text{for
all}\; r \geqslant 0.
\]
From Lemma~\ref {lemma:dualityLattices} we see that $\Delta(L_{\succ v})$
and $\Delta((\hat 0, \mu(v))_{L'})$ are isomorphic. Therefore, by~\cite
{GPWLCMLattice99}*{Theorem~2.1}, we see that
\[
\dim_\Bbbk \homology^r(U_\met, \ints/\ell) = 
\sum_{v \in L}
\beta^{R'}_{2n-r-2d(v), \deg \mu(v)}(R'/J^\vee).
\]
Write $t = 
\max\{r: \homology^r((\affine^{n+1}_\Bbbk \minus Z)_\met, \ints/\ell) \neq
0\}$. Note that $n-d(v) = \deg \mu(v)$. Hence 
\begin{align*}
t & = \max \{r : \text{there exists}\; m \in L' \;\text{such
that}\; \beta^{R'}_{2\deg m-r, \deg m}(R'/J^\vee) \neq 0\} \\
& = \max \{2j-i : \beta^{R'}_{i, j}(R'/J^\vee) \neq 0\} \\
& = \max \{2j-i : \beta^{R'}_{i, j}(R'/J^\vee) \;\text{is an extremal Betti
number of}\; (R'/J^\vee)\}.
\end{align*}
(We see the last equality as follows: Pick $i,j$ such that $\beta^{R'}_{i,
j}(R'/J^\vee) \neq 0$ and $2j-i$ is maximized. If $\beta^{R'}_{i,
j}(R'/J^\vee)$ were not an extremal Betti number, then there would be $i', j'$
such that 
\begin{inparaenum}
\item $\beta^{R'}_{i',j'}(R'/J^\vee) \neq 0$,
\item \label{enum:notTopLetCorner}
$i' \geqslant i$ and $j'-i' \geqslant j-i$, and,
\item at least one of inequalities in \eqref{enum:notTopLetCorner} is strict.
\end{inparaenum}
Then $2j'-i' > 2j-i$, a contradiction.)

Now,~\cite {BCPExtremal99}*{Theorem~2.8} gives
the following: if
$\beta^{R'}_{i,j}(R'/J^\vee)$ is an extremal Betti number of $R'/J^\vee$, 
then
$\beta^{R'}_{j-i,j}(J) = \beta^{R'}_{i,j}(R'/J^\vee)$, and
$\beta^{R'}_{j-i,j}(J)$ is an extremal Betti number of $J$. Therefore
\begin{align*}
t & = \max \{j+i : \beta^{R'}_{i, j}(R'/J^\vee) \;\text{is an extremal Betti
number of}\; (R'/J^\vee)\} \\
& = \max \{j+i : \beta^{R'}_{i, j}(R'/J^\vee) \neq 0\}.
\end{align*}
This proves the theorem.
\end{proof}

In~\cite {GPWSubspArr02}*{Theorem~3.1},
Gasharov, Peeva and Welker 
prove an
analogue of the theorem for the singular cohomology groups of affine real
subspace arrangements.

\subsection*{Reisner's ideal and the conjecture of Lyubeznik}

We present a counterexample to Conjecture~\ref{conj:lyubeznikECD}
which asserts that 
for any $\Bbbk$-scheme $U$, $\etcd(U) \geqslant \dim U + \qccd(U)$. 
We were unable to find such an example in the literature.

\begin{example}
\label{example:reisnerLyubeznik}
Let $\charact \Bbbk = 2$. Let $R
= \Bbbk[x_0, \ldots, x_5]$, $I$ the monomial $R$-ideal of the minimal
triangulation
of $\mathbb{RP}^2$ given by G.~Reisner; see, \textit{e.g.},~\cite
{BrHe:CM}*{Section~5.3}. Then $\height I = 3$, $\projdim R/I = 4$ and
$\arank I = 4$~\cite {YanEtaleGorMcp00}*{Example~2}.  Let $Z$ be the closed
$\Bbbk$-scheme of $\projective^5_\Bbbk$ defined by $I$ and 
set $U = \projective^5_\Bbbk \minus Z$. Write $\calI$ for the ideal sheaf of
$Z$.
Note that $Z$ is a $2$-dimensional
projective scheme.

Since $\projdim_R R/I = 4$, $\cohdim(I, R)  = 4$
by~\cite{LyubRegSeqLOCO84}*{Theorem~1}.
Thus, by~\eqref{equation:cohdimqccd}, $\qccd (U) = 3$.
Hence $\dim U + \qccd(U) = 8$.
We apply
Theorem~\ref{thm:subschLSTCI} to show that $\etcd_\ell(U) \leqslant 7$ for
all prime numbers $\ell \neq 2$. Let $\ell \neq 2$ be a prime number.

To show that the
hypothesis~\ref{thm:subschLSTCI}.\eqref{enum:localnminustwo} holds, it
suffices to show that the stalk $\calI_p$ is a
set-theoretic complete intersection (\textit{i.e.}, $\arank \calI_p =
\height \calI_p$)
for all $p \in Z$. Let $p \in Z$; without loss of generality,
$x_0 \neq 0$ at $p$. We will show that $IR_{x_0}$ is a 
set-theoretic complete intersection.
This follows from noting that
$(I \!:_R\! x_0) = \sqrt{(x_1x_2, x_5x_1+x_3x_4, x_4x_5+x_2x_3)}$ (see,
\textit{e.g.},~\cite {BariLinAlg08}*{p.~4545}). For
hypothesis~\ref{thm:subschLSTCI}.\eqref{enum:nonemptyintersections}, note 
that the irreducible
components are linear subspaces of $\projective^5_\Bbbk$ 
(so they are analytically
irreducible) and that they have pairwise non-empty intersection (for
instance, consider the triangulation given in \cite
{BrHe:CM}*{Section~5.3}).
To see that hypothesis~\ref{thm:subschLSTCI}.\eqref{enum:constecd} is
satisfied,
let $R'$ and $J$ be as in Notation~\ref{notation:AffineECD}.
Write $Z'$ for the affine cone over $Z$.
Since $J$ is a height three Cohen--Macaulay ideal generated by cubics
and has a linear resolution, it has a single extremal Betti number
$\beta_{2,5}(J)$, so, by 
Theorem~\ref{thm:constEtcd}.
$\homology^7((\affine^6_\Bbbk \minus Z')_\met, \ints/\ell)
\neq 0 = \homology^{\geqslant 8}((\affine^6_\Bbbk \minus Z')_\met, \ints/\ell)$.
By~\eqref{equation:affineConeEtCD},
$\homology^6(U_\met, \ints/\ell)
\neq 0 = \homology^{\geqslant 7}(U_\met, \ints/\ell)$.
Now, by Theorem~\ref{thm:subschLSTCI}, $\etcd_\ell(U) \leqslant 7$.
\end{example}

\begin{remarkwithoutbox}
\label{rem:2vs3}
\begin{asparaenum}
\item Theorems~5.3.(ii), and~6.3.(ii) of~\cite {LyubEtCohDim93} reach the 
conclusion of Lemma~\ref{lemma:vertexIrredHypers}, 
but with stronger hypotheses, which do not apply to
Example~\ref{example:reisnerLyubeznik}. 

\item \label{enum:reisnerLyubExactECD}
On Example~\ref{example:reisnerLyubeznik}:
In view of the fact that
$\homology^{7}(U_\met, \ints/\ell) = 0$ one might wonder whether
$\etcd(U_\met)$ equals $6$. However, 
let $Z_1$ be the union of all the $2$-dimensional coordinate
subspaces of $\projective^5_\Bbbk$; hence $Z_1$ is defined by 
$\bigcap_{0 \leqslant i < j < k \leqslant 5}(x_i,x_j,x_k)$. 
Theorem~\ref{thm:constEtcd}, together 
with~\eqref{equation:affineConeEtCD},
implies that $\homology^7((\projective^5_\Bbbk \minus Z_1), \ints/\ell)
\neq 0 = 
\homology^{\geqslant 8}((\projective^5_\Bbbk \minus Z_1), \ints/\ell)$.
Since $Z \subseteq Z_1$, we see that $\etcd_\ell(U) = 7$.

\item \label{rem:2vs3detail}
More on Example~\ref{example:reisnerLyubeznik}:
Instead of taking $\charact \Bbbk = 2$, we let $\charact \Bbbk = p > 0$.
(As usual, $\ell \neq p$ is a prime number.)
Let $Z \subseteq \projective^5_\Bbbk$ and 
$Z' \subseteq \affine^6_\Bbbk$ be defined by the monomial $R$-ideal $I$ of
the minimal triangulation of $\mathbb{RP}^2$. Let $J$ be as in 
Notation~\ref{notation:AffineECD}.  The extremal Betti numbers of
$J$ are
\begin{alignat*}{2}
\beta_{2,6}(J) &\;\text{and}\; \beta_{3,6}(J), &\qquad\text{if}\; \ell = 2; \\
\beta_{2,5}(J), &&\qquad\text{if}\; \ell > 2.
\end{alignat*}
Suppose that $p = 2$ (and $\ell > 2$). Then
{\makeatletter \def\theenumii{\@roman\c@enumii} \makeatother
\begin{compactenum}
\item $\homology^7(\affine^6 \minus Z', \ints/\ell) \neq 0 =
\homology^{\geqslant 8}(\affine^6 \minus Z', \ints/\ell)$ and
$\homology^6(\PP^5\minus Z,\ints/\ell) \neq 0
= \homology^{\geqslant 7}(\PP^5\minus Z,\ints/\ell)$.
(Example~\ref{example:reisnerLyubeznik}.)
\item  $\etcd(\affine^6\minus Z')=\etcd_\ell(\affine^6\minus Z')=8$ and
$\etcd(\PP^5\minus Z)=\etcd_\ell(\PP^5\minus Z)=7$.
(See~\eqref{enum:reisnerLyubExactECD} above.)
\end{compactenum}
On the other hand, if $p>2$ then
{\makeatletter\def\theenumii{\@roman\c@enumii${}^\prime$}\makeatother
\begin{compactenum}
\item 
$\homology^9(\affine^6 \minus Z', \ints/2) \neq 0 =
\homology^{\geqslant 10}(\affine^6 \minus Z', \ints/2)$ and
$\homology^8(\PP^5\minus Z,\ints/2) \neq 0
= \homology^{\geqslant 9}(\PP^5\minus Z,\ints/2)$. 
(Theorem~\ref{thm:constEtcd} gives the affine case; the projective case then
follows from~\eqref{equation:affineConeEtCD}.)
\item
$\etcd(\affine^6\minus Z') = \etcd_2(\affine^6\minus Z')=9$ and 
$\etcd(\PP^5 \minus Z) = \etcd_2(\PP^5\minus Z)=8$.
(The projective case follows from Theorem~\ref{thm:subschLSTCI}, and it
gives the affine case by~\eqref{equation:affineConeCDs}.) 
\end{compactenum}}}
In particular, this example shows that topological data on closed
subschemes and combinatorial data on their intersections do not suffice to
predict the \etale cohomological dimension of the complement of their
union. \hfill\qedsymbol
\end{asparaenum}
\end{remarkwithoutbox}

\begin{remark}
\label{remark:hypForProjComplement}
Hypothesis~\ref{thm:subschLSTCI}.\eqref{enum:localnminustwo} is
crucial to the proof. Even with
it, the other two hypotheses 
are independent of each other. In
Remark~\ref{rem:2vs3}.\eqref{rem:2vs3detail}, we saw an example (with $p >
2$ and $\ell =2$) in which 
the hypothesis~\ref{thm:subschLSTCI}.\eqref{enum:nonemptyintersections}
holds, but 
the hypothesis~\ref{thm:subschLSTCI}.\eqref{enum:constecd} does not.

The following is an example
where~\ref{thm:subschLSTCI}.\eqref{enum:constecd} is satisfied,
but~\ref{thm:subschLSTCI}.\eqref{enum:nonemptyintersections} is not.
Let $Z \subseteq \projective^5_\Bbbk$ be defined by $I =
(x_0x_1, x_2x_3, x_4x_5, x_0x_3, x_0x_5, x_2x_5)$, and let $Z'$ be the
affine cone over $Z$. (There are no restrictions on $\charact \Bbbk$ and
$\ell$, except that $\ell \neq \charact \Bbbk$.)
Then $Z$ is a codim $3$
set-theoretic complete intersection (see, \textit{e.g.},~\cite
{KummBipIdeals08}*{Theorem 1.3}). Its two irreducible components defined by
$(x_0, x_2, x_4)$ and $(x_1, x_3, x_5)$ are disjoint. The ideal $J$ has a
linear resolution, so it has one extremal Betti number $\beta_{2,4}(J)$. 
Hence
$\homology^{\geqslant 7}((\affine^6 \minus Z')_\met, \ints/\ell) = 0$, so
$\homology^{\geqslant 6}((\projective^n \minus Z)_\met, \ints/\ell)  =0$.
\end{remark}

\section{Upper bounds from the lcm-lattice}
\label{sec:upperbounds}

Let $L$ be the lcm-lattice of $I$.
Using rooting complexes (see~\cite{BjoZieBrokenCkts91}*{Section~3}
and~\cite{NovikRooted02}) we obtain an upper bound for $\arank_h I$.  Write
$L_1$ for the set of atoms (\textit{i.e.}, elements of height one) of $L$;
this is the set of monomial minimal generators of $I$.  For $F \subseteq
L_1$, let $\lambda_F = \lcm F$.  For a simplicial complex $\Gamma$, denote
the set of $r$-dimensional faces by $\Gamma^{(r)}$.

\begin{defn}
\label{defn:rootingmaps}
A \emph{rooting map} is a function $\rho : L \rightarrow L_1$ such that for
every $m \in L$, $\rho(m) | m$ and for all $m' \in [\rho(m), m], \rho(m') =
\rho(m)$. 
We say that $G \subseteq L_1$ is \emph{unbroken} if $\rho(\lambda_G) \in
G$. The \emph{rooting complex} of $L$ corresponding to $\rho$ is $\Gamma_{I,
\rho} = \{ F \subseteq L_1 : G \;\text{is unbroken, for all}\; G \subseteq
F\}$. 
\end{defn}

\begin{lemma}[\protect{\cite{BjoZieBrokenCkts91}*{Theorem~3.2(2)},~\cite%
{NovikRooted02}*{Proposition~1(1)}}]
\label{thm:rootCxIsCone}
The simplicial complex
$\Gamma_{I,\rho}$ is a cone with apex $\rho(\lambda_{L_1})$.
\end{lemma}

%

\begin{lemma}
\label{thm:rooCxVarRestr}
Let $I' = (I \cap \Bbbk[x_1, \ldots, \widehat{x_r}, \ldots, x_n])R$. Let
$L'$ be the lcm-lattice of $I'$, $L'_1$, the set of atoms of $L'$, and $\rho' := \rho|_{L'}$. Then
\begin{inparaenum}
\item \label{enum:rootMapRestr} $\rho'$ is a rooting map of $L'$, and
\item \label{enum:rootCxRes} $\Gamma_{I', \rho'} = \Gamma_{I, \rho}
\big|_{L'_1}$.
\end{inparaenum}
\end{lemma}

\begin{proof}
The proof of \eqref{enum:rootMapRestr} is immediate. To prove
\eqref{enum:rootCxRes}, we need to show that for all $F \subseteq L'_1$,
$F$ is a face of $\Gamma_{I', \rho'}$  if and only if $F$ is a face of
$\Gamma_{I, \rho}$. Let $F \subseteq L'_1$. Then $F$ is a face of
$\Gamma_{I', \rho'}$ if and only if $\rho'(\lambda_G) \in G$ for all $G
\subseteq F$, which holds if and only if $\rho(\lambda_G) \in G$ for all $G
\subseteq F$, which holds if and only if $F$ is a face of
$\Gamma_{I, \rho}$.
\end{proof}

The following theorem generalizes a result of
K.~Kimura~\cite{KimuraLyubResARA09}*{Theorem~1}, which shows that $\arank
I$ is at most the length of any Lyubeznik resolution
(see~\cite{LyubMonFreeRes88}). I.~Novik~\cite{NovikRooted02}*{Remark after
Theorem~1} showed that Lyubeznik resolutions arise as special cases of
resolutions supported on rooting complexes. More precisely, rooting maps
induce total orders on each rooted set; if these total orders are
compatible with each other and induce a total order on $\gens(I)$, then the
corresponding free resolution is a Lyubeznik resolution.
\begin{thm}
\label{thm:araLeqDimRootingCx}
Fix an integer $d \geqslant \max\{\deg f: f \in \gens(I)\}$. Let
$\ol{(-)} : \gens(I) \longrightarrow \monom(R)$ be a function such
that $\deg \ol{f} = d$ and $\sqrt{\ol{f}} =
\sqrt{f}$, for all $f \in \gens(I)$. For $r=0, \ldots, \dim
\Gamma_{I,\rho}$, let 
\[
g_r = \sum_{F \in \left(\Gamma_{I, \rho}\right)^{(r)}} \prod_{f \in F}
\ol{f}.
\]
Then $I = \sqrt{(g_i: 0 \leqslant r \leqslant \dim \Gamma_{I, \rho})}$.
In particular, $\arank_h I \leqslant 1 + \dim \Gamma_{I, \rho}$.
\end{thm}

\begin{proof}
The second assertion follows from the first, after noting that for each
$r$, $g_r$ is a homogeneous polynomial of degree $(r+1)d$. We prove the
first assertion by induction on $n=\dim R$. We may assume that it holds for
all monomial ideals generated in $n-1$ or fewer variables.

Write $\Gamma = \Gamma_{I,\rho}$, $l = \dim \Gamma_{I, \rho}$ and $J =
(g_0, \ldots, g_l)$. For $F \subseteq L_1$, write $\pi_F = \prod_{f \in F}
\ol f$. Without loss of generality, $\rho(\lambda_{L_1}) =
f_1$. It follows from Lemma~\ref{thm:rootCxIsCone} that 
\begin{equation}
\label{eqn:glRecursion}
g_r  = \begin{cases}
\ol{f_1} + \cdots + \ol{f_m}, & r = 0 \\
\ol{f_1}
\left(\sum\limits\limits_{\substack{F \in \Gamma^{(r-1)} \\ f_1 \not
\in F}} \pi_F \right) + 
\sum\limits_{\substack{F \in \Gamma^{(r)}  \\ f_1 \not \in F}} 
\pi_F, & 1 \leqslant r \leqslant l-1 \\
\ol{f_1}
\left(\sum\limits\limits_{\substack{F \in \Gamma^{(l-1)} \\ f_1 \not
\in F}} \pi_F \right), & r = l.
\end{cases}
\end{equation}

Let $\mathfrak p \in \Spec R$ be such that $J \subseteq \mathfrak p$. We
first claim that $\ol{f_1} \in \mathfrak p$. For, otherwise, it follows
from \eqref{eqn:glRecursion} that $\ol{f_2} + \cdots + \ol{f_r} \in
\mathfrak p$, which implies that $\ol{f_1} \in \mathfrak p$, a
contradiction; hence $\ol{f_1} \in \mathfrak p$. Therefore there exists $i$
such that $x_i \in \mathfrak p$. 
Then $(J, x_i) \subseteq \mathfrak p$.
Let $I' = (I \cap \Bbbk[x_1, \ldots, \widehat{x_i}, \ldots,
x_n])R$.  Let $L'$ be the lcm-lattice of $I'$, $L'_1$ the set of atoms of
$L'$, and $\rho' := \rho|_{L'}$. 
Since 
\[
g_r \equiv 
\sum_{F \in (\Gamma_{I', \rho'})^{(r)}} 
\prod_{f \in F} \ol{f} \quad \mod x_i, 
\]
we see, by induction, that $\sqrt{(J,x_i)} = (I',x_i)$. Therefore $I
\subseteq \mathfrak p$.
\end{proof}

Since Lyubeznik resolutions arise as special cases of resolutions supported
on rooting complexes, one might wonder whether it is possible to construct
shorter resolutions by taking rooting maps that do not induce a global
ordering of $\gens(I)$. This is the motivation of the following question.
We, however, do not know the answer.
\begin{question}
Let $L$ be any finite atomic lattice. Is
\[
\min \{\dim \Gamma_{L, \rho} : \rho \;\text{a rooting map}\} \gneq
\min \{\dim \Gamma_{L, \rho} : \rho \;\text{a rooting map induced by a
total order on}\; L_1\}?
\]
\end{question}

The polynomials $g_r$ in the previous theorem may be of relatively large
degree. The following proposition provides another set of generators up to
radical, of smaller degree, at the expense of increasing the number of
generators.

\begin{propn}
\label{thm:araLeqHtLcmLattice}
For $r=1, \ldots, \height L$, fix an integer $d_r \geqslant  \max \{\deg \sigma
: \sigma \in L, \height \sigma = r\}$. Let $\ol{(\;)} : L
\longrightarrow \monom(R)$ be a function such that $\deg \ol{\sigma}
= d_{\height
\sigma}$ and
$\sqrt{\ol{\sigma}} = \sqrt{\sigma}$, for all $\sigma \in L$.
For $r=1, \ldots, \height L$, let
\[
g_r = \sum_{\substack{\sigma \in L \\ \height \sigma = r}}
\ol{\sigma}.
\]
Then $I = \sqrt{(g_i: 1 \leqslant r \leqslant \height L)}$.
In particular, $\arank_h I \leqslant 1 + \height L$.
\end{propn}

\begin{proof}
It suffices to show that $I \subseteq \sqrt{(g_i: 1 \leqslant r \leqslant \height
L)}$, which we do by induction on $n=\dim R$. We may assume that the
assertion holds for all monomial ideals generated in $n-1$ or fewer
variables.

Write $l = \height L$ and $J = (g_1, \ldots, g_l)$. Note that $g_l =
x_1x_2\cdots x_n$. Let $\mathfrak p \in \Spec R$ be such that $J \subseteq
\mathfrak p$. Then, there exists $t$ such that $(J, x_t) \subseteq
\mathfrak p$. Let $I'$ be the subideal of $I$ generated by the minimal
generators of $I$ that are not divisible by $x_t$. Write $L'$ for the the
lcm-lattice of $I'$. It is a sublattice of $L$, and, furthermore, for any
$\sigma \in L'$, $\height_L \sigma = \height_{L'} \sigma$. Hence, by
induction, $\sqrt{(J,x_t)} = (I', x_t) = (I,x_t)$. Therefore $I \subseteq
\mathfrak p$.
\end{proof}

\section*{Acknowledgements}
We thank D.~Arapura and G.~Lyubeznik for helpful comments.
The computer algebra system \texttt{Macaulay2} \cite{M2} provided valuable
assistance in studying examples.


\begin{bibdiv}
\begin{biblist}

\bib{BariLinAlg08}{article}{
      author={Barile, Margherita},
       title={Arithmetical ranks of {S}tanley-{R}eisner ideals via linear
  algebra},
        date={2008},
        ISSN={0092-7872},
     journal={Comm. Algebra},
      volume={36},
      number={12},
       pages={4540\ndash 4556},
         url={http://dx.doi.org/10.1080/00927870802182614},
      review={\MR{2473347 (2009h:13030)}},
}

\bib{BCPExtremal99}{article}{
      author={Bayer, Dave},
      author={Charalambous, Hara},
      author={Popescu, Sorin},
       title={Extremal {B}etti numbers and applications to monomial ideals},
        date={1999},
        ISSN={0021-8693},
     journal={J. Algebra},
      volume={221},
      number={2},
       pages={497\ndash 512},
         url={http://dx.doi.org/10.1006/jabr.1999.7970},
      review={\MR{1726711 (2001a:13020)}},
}

\bib{BrHe:CM}{book}{
      author={Bruns, Winfried},
      author={Herzog, J{\"u}rgen},
       title={Cohen-{M}acaulay rings},
      series={Cambridge Studies in Advanced Mathematics},
   publisher={Cambridge University Press},
     address={Cambridge},
        date={1993},
      volume={39},
        ISBN={0-521-41068-1},
      review={\MR{1251956 (95h:13020)}},
}

\bib{BrunsSchwanzlDetVar90}{article}{
      author={Bruns, Winfried},
      author={Schw{\"a}nzl, Roland},
       title={The number of equations defining a determinantal variety},
        date={1990},
        ISSN={0024-6093},
     journal={Bull. London Math. Soc.},
      volume={22},
      number={5},
       pages={439\ndash 445},
         url={http://dx.doi.org/10.1112/blms/22.5.439},
      review={\MR{MR1082012 (91k:14035)}},
}

\bib{BrShLoco98}{book}{
      author={Brodmann, M.~P.},
      author={Sharp, R.~Y.},
       title={Local cohomology: an algebraic introduction with geometric
  applications},
      series={Cambridge Studies in Advanced Mathematics},
   publisher={Cambridge University Press},
     address={Cambridge},
        date={1998},
      volume={60},
        ISBN={0-521-37286-0},
      review={\MR{MR1613627 (99h:13020)}},
}

\bib{BjoZieBrokenCkts91}{article}{
      author={Bj{\"o}rner, Anders},
      author={Ziegler, G{\"u}nter~M.},
       title={Broken circuit complexes: factorizations and generalizations},
        date={1991},
        ISSN={0095-8956},
     journal={J. Combin. Theory Ser. B},
      volume={51},
      number={1},
       pages={96\ndash 126},
         url={http://dx.doi.org/10.1016/0095-8956(91)90008-8},
      review={\MR{MR1088629 (92b:52027)}},
}

\bib{GPWSubspArr02}{incollection}{
      author={Gasharov, Vesselin},
      author={Peeva, Irena},
      author={Welker, Volkmar},
       title={Coordinate subspace arrangements and monomial ideals},
        date={2002},
   booktitle={Computational commutative algebra and combinatorics ({O}saka,
  1999)},
      series={Adv. Stud. Pure Math.},
      volume={33},
   publisher={Math. Soc. Japan},
     address={Tokyo},
       pages={65\ndash 74},
      review={\MR{1890096 (2003c:13011)}},
}

\bib{GPWLCMLattice99}{article}{
      author={Gasharov, Vesselin},
      author={Peeva, Irena},
      author={Welker, Volkmar},
       title={The lcm-lattice in monomial resolutions},
        date={1999},
        ISSN={1073-2780},
     journal={Math. Res. Lett.},
      volume={6},
      number={5-6},
       pages={521\ndash 532},
      review={\MR{MR1739211 (2001e:13018)}},
}




\bib{Loco24Hrs}{book}{
      author={Iyengar, Srikanth~B.},
      author={Leuschke, Graham~J.},
      author={Leykin, Anton},
      author={Miller, Claudia},
      author={Miller, Ezra},
      author={Singh, Anurag~K.},
      author={Walther, Uli},
       title={Twenty-four hours of local cohomology},
      series={Graduate Studies in Mathematics},
   publisher={American Mathematical Society},
     address={Providence, RI},
        date={2007},
      volume={87},
        ISBN={978-0-8218-4126-6},
      review={\MR{MR2355715}},
}

\bib{KimuraLyubResARA09}{article}{
      author={Kimura, Kyouko},
       title={Lyubeznik resolutions and the arithmetical rank of monomial
  ideals},
        date={2009},
        ISSN={0002-9939},
     journal={Proc. Amer. Math. Soc.},
      volume={137},
      number={11},
       pages={3627\ndash 3635},
         url={http://dx.doi.org/10.1090/S0002-9939-09-09950-X},
      review={\MR{MR2529869 (2010f:13013)}},
}

\bib{KummBipIdeals08}{article}{
      author={Kummini, Manoj},
       title={Regularity, depth and arithmetic rank of bipartite edge ideals},
        date={2009},
        ISSN={0925-9899},
     journal={J. Algebraic Combin.},
      volume={30},
      number={4},
       pages={429\ndash 445},
         url={http://dx.doi.org/10.1007/s10801-009-0171-6},
      review={\MR{MR2563135}},
}

\bib{LyubSurveyLocalCoh02}{incollection}{
      author={Lyubeznik, Gennady},
       title={A partial survey of local cohomology},
        date={2002},
   booktitle={Local cohomology and its applications ({G}uanajuato, 1999)},
      series={Lecture Notes in Pure and Appl. Math.},
      volume={226},
   publisher={Dekker},
     address={New York},
       pages={121\ndash 154},
      review={\MR{MR1888197 (2003b:14006)}},
}

\bib{LyubRegSeqLOCO84}{incollection}{
      author={Lyubeznik, Gennady},
       title={On the local cohomology modules {$H\sp i\sb {{\mathfrak a}}(R)$}
  for ideals {${\mathfrak a}$} generated by monomials in an {$R$}-sequence},
        date={1984},
   booktitle={Complete intersections (acireale, 1983)},
      series={Lecture Notes in Math.},
      volume={1092},
   publisher={Springer},
     address={Berlin},
       pages={214\ndash 220},
      review={\MR{MR775884 (86f:14002)}},
}

\bib{LyubMonFreeRes88}{article}{
      author={Lyubeznik, Gennady},
       title={A new explicit finite free resolution of ideals generated by
  monomials in an {$R$}-sequence},
        date={1988},
        ISSN={0022-4049},
     journal={J. Pure Appl. Algebra},
      volume={51},
      number={1-2},
       pages={193\ndash 195},
         url={http://dx.doi.org/10.1016/0022-4049(88)90088-6},
      review={\MR{MR941900 (89c:13020)}},
}

\bib{LyubEtCohDim93}{article}{
      author={Lyubeznik, Gennady},
       title={\'{E}tale cohomological dimension and the topology of algebraic
  varieties},
        date={1993},
        ISSN={0003-486X},
     journal={Ann. of Math. (2)},
      volume={137},
      number={1},
       pages={71\ndash 128},
         url={http://dx.doi.org/10.2307/2946619},
      review={\MR{1200077 (93m:14016)}},
}

\bib{M2}{misc}{ label={M2},
      author={Grayson, Daniel~R.},
      author={Stillman, Michael~E.},
       title={Macaulay 2, a software system for research in algebraic
  geometry},
        date={2006},
        note={Available at \href{http://www.math.uiuc.edu/Macaulay2/}
  {http://www.math.uiuc.edu/Macaulay2/}},
}

\bib{MatsCRT89}{book}{
      author={Matsumura, Hideyuki},
       title={Commutative ring theory},
     edition={Second},
      series={Cambridge Studies in Advanced Mathematics},
   publisher={Cambridge University Press},
     address={Cambridge},
        date={1989},
      volume={8},
        ISBN={0-521-36764-6},
        note={Translated from the Japanese by M. Reid},
      review={\MR{MR1011461 (90i:13001)}},
}

\bib{MilnEtale80}{book}{
      author={Milne, James~S.},
       title={\'{E}tale cohomology},
      series={Princeton Mathematical Series},
   publisher={Princeton University Press},
     address={Princeton, N.J.},
        date={1980},
      volume={33},
        ISBN={0-691-08238-3},
      review={\MR{559531 (81j:14002)}},
}

\bib{NovikRooted02}{article}{
      author={Novik, Isabella},
       title={Lyubeznik's resolution and rooted complexes},
        date={2002},
        ISSN={0925-9899},
     journal={J. Algebraic Combin.},
      volume={16},
      number={1},
       pages={97\ndash 101},
         url={http://dx.doi.org/10.1023/A:1020838732281},
      review={\MR{MR1941987 (2003j:13021)}},
}

\bib{OgusLocCohDim73}{article}{
      author={Ogus, Arthur},
       title={Local cohomological dimension of algebraic varieties},
        date={1973},
        ISSN={0003-486X},
     journal={Ann. of Math. (2)},
      volume={98},
       pages={327\ndash 365},
      review={\MR{MR0506248 (58 \#22059)}},
}

\bib{SGA4}{book}{ label={SGA4},
       title={Th\'eorie des topos et cohomologie \'etale des sch\'emas.},
      series={Lecture Notes in Mathematics, Vols. 269, 270 and 305},
   publisher={Springer-Verlag},
     address={Berlin},
        date={1972-3},
        note={S{\'e}minaire de G{\'e}om{\'e}trie Alg{\'e}brique du Bois-Marie
  1963--1964 (SGA 4), Dirig{\'e} par M. Artin, A. Grothendieck, et J. L.
  Verdier. Avec la collaboration de N. Bourbaki, P. Deligne et B. Saint-Donat},
}

\bib{VarbaroArithRk10}{misc}{
      author={Varbaro, Matteo},
       title={On the arithmetical rank of certain {S}egre embeddings},
        date={2010},
        note={arXiv:1007.5440v1 [math.AG]},
}

\bib{YanEtaleGorMcp00}{article}{
      author={Yan, Zhao},
       title={An \'etale analog of the {G}oresky-{M}ac{P}herson formula for
  subspace arrangements},
        date={2000},
        ISSN={0022-4049},
     journal={J. Pure Appl. Algebra},
      volume={146},
      number={3},
       pages={305\ndash 318},
      review={\MR{MR1742346 (2000k:14041)}},
}

\end{biblist}
\end{bibdiv}

\def\cfudot#1{\ifmmode\setbox7\hbox{$\accent"5E#1$}\else
  \setbox7\hbox{\accent"5E#1}\penalty 10000\relax\fi\raise 1\ht7
  \hbox{\raise.1ex\hbox to 1\wd7{\hss.\hss}}\penalty 10000 \hskip-1\wd7\penalty
  10000\box7}

\end{document}